\documentclass[12pt,dvips]{article}
\usepackage[T1]{fontenc}
\usepackage{lmodern}
\usepackage{geometry}
\usepackage{amssymb}
\usepackage{array}
\usepackage{epic,eepic}
\usepackage{graphicx}

\newtheorem{theorem}{Theorem}
\newtheorem{lemma}[theorem]{Lemma}

\newtheorem{observation}[theorem]{Observation}

\newtheorem{definition}[theorem]{Definition}

\newenvironment{proof}%
{\begin{trivlist}\item[]{\bf Proof:} }%
{\mbox{~~}\hfill$\Box$\end{trivlist}}

\newenvironment{example}%
{\begin{trivlist}\item[]{\bf Example:} }%
{\mbox{~~}\hfill$\Box$\end{trivlist}}

\newcommand{\rn}{{\rm RN}}
\newcommand{\ru}{{\rm RU}}
\newcommand{\rd}{{\rm RD}}
\newcommand{\rz}{{\rm RZ}}
\newcommand{\ra}{{\rm RA}}
\newcommand{\tr}{{\rm tr}}

\newcommand{\txt}{\textstyle}

\newcommand{\mb}{\makebox(0,0)}

\newcommand{\ulp}{{\rm{ulp}}}

\newcommand{\eps}{\varepsilon}

\newcommand{\lzd}{{\tt lzd}}

\newcommand{\val}{{\mathcal{V}}}
\newcommand{\inv}{{\mathcal{I}}}

\newcommand{\specialdescriptionlabel}[1]{\bf #1\hfil}

\title{\bf Floating Point Arithmetic on\\ Round-to-Nearest Representations}

\author{%
\begin{tabular}{c}
Peter~Kornerup\\
\textit{SDU, Odense, Denmark} \\
\end{tabular}
\begin{tabular}{c}
Jean-Michel~Muller\hspace{5mm}
Adrien  Panhaleux\\
\textit{LIP, CNRS/ENS Lyon, France} \\
\end{tabular}}

\date{October 7, 2011}

\begin{document}

\maketitle

\begin{abstract}
  Recently we introduced a class of number representations denoted
  RN-representations, allowing an un-biased rounding-to-nearest to take
  place by a simple truncation. In this paper we briefly review the binary
  fixed-point representation in an encoding which is essentially an
  ordinary 2's~complement representation with an appended round-bit. Not
  only is this rounding a constant time operation, so is also sign
  inversion, both of which are at best log-time operations on ordinary
  2's~complement representations.  Addition, multiplication and division is
  defined in such a way that rounding information can be carried along in a
  meaningful way, at minimal cost. Based on the fixed-point encoding we
  here define a floating point representation, and describe to some detail
  a possible implementation of a floating point arithmetic unit employing
  this representation, including also the directed roundings.
\end{abstract}

\section{Introduction}
In \cite{KM05b} a class of number representations denoted RN-Codings were
introduced, ``RN'' standing for ``round-to-nearest'', as these
radix-$\beta$, signed-digit representations have the property that
truncation yields rounding to the nearest representable value.  They are
based on a generalization of the observation that certain radix
representations are known to posses this property, e.g., the balanced
ternary ($\beta=3$) system over the digit set $\{-1,0,1\}$.  Another such
representation is obtained by performing the original Booth-recoding
\cite{Boo51} on a 2's~complement number into the digit set $\{-1,0,1\}$,
where it is well-known that the non-zero digits of the recoded number
alternate in sign. Besides the simplicity of rounding by truncation, it has
the feature that the effect of one rounding followed by another rounding
yields the same result, as would be obtained by a single rounding to the
same precision as the last. It is a known problem when using any ``extended
precision'' of the IEEE-754 standard \cite{IEEE08}, performing computations
in ``extended-80'' format representation (e.g., the Intel extended double
precision) storing the result in the binary ``basic-64'' format.

When we are not concerned with the actual encoding of a value, we shall
here use the notation {\em RN-representation}. This representation was
further discussed in \cite{KMP11}, where a special encoding of the
Booth-recoded binary representation was introduced. This encoding, termed
the {\it canonical} encoding, is based on the ordinary 2's~complement
representation, but with an appended round-bit, allowing round-to-nearest
by truncation. Arithmetic on operands in this canonical encoding is
essentially standard 2's~complement arithmetic, but with the added benefit
that negation is a constant time operation, obtained by bit inversion.

To be able to discuss a possible implementation of a floating point
arithmetic unit it is necessary to repeat some material on the fixed point
representation from previous publications. We shall in Section~2 (adapted
from \cite{KM05b}) cite some definitions, however here restricted to the
binary representation. Section~3 citing some material from \cite{KMP11} but
also expanding it, analyzes the relation between RN-representations and
2's~complement representations. Conversion from the latter into the former
is performed by the Booth algorithm, yielding a signed-digit representation
in a straightforward encoding. It is then realized that $n+1$ bits are
sufficient, providing a simple encoding consisting of the $n$ bits of the
2's~complement encoding with a round-bit appended, yielding the {\it
  canonical} encoding. To be able to define arithmetic operations directly
on the canonically encoded numbers it turns out to be useful to interpret
them as intervals, reflecting the sign of what has possibly been rounded
away by truncation.

Section~4 then presents implementations of addition, multiplication and
division on fixed-point, canonically encoded RN-represented numbers, to
some extent repeating material from \cite{KMP11}, but also expanding the
descriptions on multiplication and division. It is noticed that these
implementations are essentially identical to standard 2's~complement
arithmetic. After introducing a floating point representation and encodings
similar to the IEEE-754 formats, Section~5 shows that multiplication and
division on such floating point operands can be defined in a
straightforward way based on the fixed-point algorithms. Addition and
subtraction is discussed in some detail, split in the traditional ``near''
and ``far'' cases. Directed roundings are are shown realizable based on a
``sticky-bit'', but without the need for a rounding
incrementation. Section~6 finally concludes the paper.

\section{Binary RN-representations}

For an introduction to the general class of RN-representations for odd and
even radix we refer the reader to \cite{KM05b}. Here we are concentrating
on the case of radix 2 over the digit-set $\{-1,0,1\}$, with the
restriction that the signs of non-zero digits alternate.

\begin{definition}[Binary RN-representation]\label{RN.def}~\\
  The digit sequence $D=d_n d_{n-1}d_{n-2}\cdots$ (with $-1\leq d_i \leq
  1$) is a binary RN-representation of $x$ iff
\begin{enumerate}
\item $x = \sum_{i=-\infty}^{n} d_i\beta^i$ (that is $D$ is a binary
  representation of $x$);
\item for any $j \leq n$,
  \[
  \left|\sum_{i=-\infty}^{j-1}d_i2^i\right|\leq\frac12 2^{j}, 
  \]
  that is, if the digit sequence is truncated to the right at any position
  $j$, the remaining sequence is always the number (or one of the two
  members in case of a tie) of the form
  $d_nd_{n-1}d_{n-2}d_{n-3}\ldots{}d_j$ that is closest to $x$.
\end{enumerate}
\end{definition}

Hence, truncating the RN-representation of a number at any position is
equivalent to rounding it to the nearest.  Although it is possible to deal
with infinite representations, we shall restrict our discussions to finite
representations, and find for such RN-representations some observations:

\pagebreak[3]
\begin{theorem}[Binary RN-representations]~\\
\label{characterization}
$D=d_m d_{m-1}\cdots d_{\ell}$ is a binary RN-representation iff
  \begin{enumerate}
    \item all digits have absolute value less than or equal 1;
  \item if $|d_i| = 1$, then the first non-zero digit that
    follows on the right has the opposite sign, that is, the largest $j <
    i$ such that $d_j \neq 0$ satisfies $d_i \times d_j < 0$,
  \end{enumerate}
  with some numbers having two finite representations, where one has its
  least significant nonzero digit equal to $1$, the other
  one has its least significant nonzero digit equal to $-1$.
\end{theorem}

A number whose finite representation has its last nonzero digit equal to
$1$ has an alternative representation ending with $-1$. Just assume the
last two digits e.g.\ are $d1$: since the representation is an
RN-representation, if we replace these two digits by $(d+1)(-1)$ we still
have a valid RN-representation. This has an interesting consequence:
truncating a number which is a tie will round either way, depending on
which of the two possible representations the number happens to have. Note
that when a rounding has taken place, the sign of a non-zero part
rounded away will have the opposite sign of the last non-zero digit, thus
the representation carries information about what was rounded away, thus
effectively halving the error bound on the result. We shall see below that
this information may be utilized in subsequent calculations.

This rounding rule is thus different from the ``round-to-nearest-even''
rule required by the IEEE floating point standard \cite{IEEE08}. Both
roundings provide a ``round-to-nearest'' in the case of a tie, but employ
different rules when choosing which way to round in this situation. But
note that the direction of rounding in general depends on how the value to
be rounded was derived, as the representation of the value in the tie
situation is determined by the sequence of operations leading to the
value. However, when employing the canonical encoding based on a
2's~complement encoding, and the implementation of the basic arithmetic
operations later, then the rounding in the tie-situations is deterministic.



\section{Encoding Binary RN-represented Numbers}

Here we briefly cite for completeness from \cite{KMP11} the definition and
some properties of the canonical binary representation and its encoding. 
Consider a value $x=-b_m2^{m}+\sum_{i=\ell}^{m-1} b_i2^i$ in
2's~complement representation:
\[
x \sim b_m b_{m-1}\cdots b_{\ell+1}b_\ell
\]
with $b_i\in\{0,1\}$ and $m>\ell$. Then the digit string
\[
\delta_m\delta_{m-1}\cdots 
\delta_{\ell+1}\delta_\ell\quad\mbox{with}\quad \delta_i\in\{-1,0,1\}
\]
defined (by the Booth recoding \cite{Boo51}) for $i=\ell,\cdots,m$ as
\begin{equation}\label{eq-delta-def}
\delta_i=b_{i-1}-b_i\quad\mbox{(with $b_{\ell-1}=0$ by convention)}
\end{equation}
is an RN-representation of $x$ with $\delta_i\in\{-1,0,1\}$. That it
represents the same value follows trivially by observing that the converted
string represents the value $2x-x$. The alternation of the signs of
non-zero digits is easily seen by considering how strings of the form
$011\cdots10$ and $100\cdots01$ are converted.

Hence the digits of the 2's~complement representation directly provides an
encoding of the converted digits as a tuple: $\delta_i\sim(b_{i-1},b_i)$
for $i=\ell,\cdots,m$ where
\begin{equation}\label{bs-encoding}
\begin{array}{r@{\;\sim\;}l}
-1 & (0,1)\\
 0 & (0,0)\mbox{ or }(1,1)\\
 1 & (1,0),
\end{array}
\end{equation}
where the value of the digit is the difference between the first and the
second component.

\begin{example}
  Let $x=110100110010$ be a sign-extended 2's~complement number and write
  the digits of $2x$ above the digits of $x$:
  \[\renewcommand{\arraystretch}{1.3}\small
  \begin{array}{|r||*{13}{c|}}\hline
  2x              &1&    0&1&    0&0&1&1&    0&0&1&    0&0\\\hline
  x               &1&    1&0&    1&0&0&1&    1&0&0&    1&0\\\hline\hline
  \mbox{RN-repr. }x& &\bar1&1&\bar1&0&1&0&\bar1&0&1&\bar1&0\\\hline
\end{array}
\]
where it is seen that in any column the two upper-most bits provide the
encoding defined above of the signed-digit below in the column. Since the
digit in position $m\!+\!1$ will always be $0$, there is no need to include
the most significant position otherwise found in the two top rows.
\end{example}

If $x$ is non-zero and $b_k$ is the least significant non-zero bit of the
2's~complement representation of $x$, then $\delta_k=-1$, confirmed in the
example, hence the last non-zero digit is always $-1$ and thus unique.
However, if an RN-represented number is truncated for rounding somewhere,
the resulting representation may have its last non-zero digit of value $1$.

As mentioned in Theorem~\ref{characterization} there are exactly two finite
binary RN-representations of any non-zero binary number of the form $a2^k$
for integral $a$ and $k$, hence requiring a specific sign of the last
non-zero digit would make the representation unique.  On the other hand
without this requirement, rounding by truncation of the 2's~complement
encoding also makes the rounding deterministic and furthermore unbiased in
the tie-situation, by rounding up or down, depending on the sign of the
digit rounded away.

\begin{example}
  Rounding the value of $x$ in the previous example by truncating off the
  two least significant digits we obtain
  \[\renewcommand{\arraystretch}{1.3}\small
  \begin{array}{|r||*{12}{c|}}\hline
  \tr_2(2x)                 &1&    0&1&    0&0&1&1&    0&0&{\bf1}\\\hline
  \tr_2(x)                  &1&    1&0&    1&0&0&1&    1&0&0\\ \hline\hline
  \mbox{RN-repr. }\rn_2(x)& &\bar1&1&\bar1&0&1&0&\bar1&0&1\\\hline
\end{array}
\]
where it is noted that the bit of value 1 in the upper rightmost corner (in
boldface) acts as a round bit, carrying information about the part rounded
away by truncation (which here happened to be a tie situation).
\end{example}

The example shows that there is very compact encoding of RN-represented
numbers derived directly from the 2's~complement representation, noting in
the example that the upper row need not be part of the encoding, except for
the round-bit. We will denote it the {\em canonical encoding}, and note
that it is a kind of ``carry-save'' in the sense that it contains a bit not
yet added in.

\begin{definition}[Binary canonical RN-encoding]
  \label{def-canon}~\\
  Let a number $x$ be given in 2's~complement representation as the bit
  string $b_m\cdots{}b_{\ell+1}b_\ell$, such that
  $x=-b_{m}2^{m}+\sum_{i=\ell}^{m-1} b_i2^i$.  Then the {\em binary canonical
    encoding} of the RN-representation of $x$ is defined as the pair
  \[
  x\sim(b_m b_{m-1}\cdots{}b_{\ell+1}b_\ell,\mbox{r})
          \;\mbox{\rm where the round-bit is }\mbox{r}=0
  \]
  and after truncation at position $k$, for $m\ge k>\ell$
  \[
  \rn_k(x)\sim(b_m b_{m-1}\cdots{}b_{k+1}b_k,r)
          \;\mbox{\rm with round-bit}\;r=b_{k-1}.
  \]
\end{definition}

The signed-digit interpretation is available with digits $\delta_i$ from
the canonical encoding by pairing bits using the encoding
(\ref{bs-encoding}), $(b_{i-1},b_i)$ as $\delta_i=b_{i-1}-b_i$ for $i>k$,
and $({r},b_k)$ with $\delta_k=r-b_k$, when truncated at position~$k$.

The fundamental idea of the canonical radix-2 RN-encoding is that it is a
binary encoding of a value represented in the signed digit set
$\{-1,0,1\}$, where the non-zero digits alternate in sign.  Using this
encoding of such numbers employing 2's complement representation in the
form $(a,r_a)$ it is seen that it then represents the value
\[
             (2a + r_au) - a = a + r_au,
\]
where $u$ is the weight of the least significant position of $a$. Note that
there is then no difference between $(a,1)$ and $(a+u,0)$, both being
RN-representations of the same value: 
\[
\forall a, \val(a,1) = \val(a+u,0),
\] 
where we use the notation $\val(x,r_x)$ to denote the value of an
RN-represented number. 

If $(x,r_x)$ is the binary canonical encoding of $X=\val(x,r_x)=x+r_xu$
then it follows that
\[
-X=-x-r_xu=\bar{x}+u-r_xu=\bar{x}+(1-r_x)u=\bar{x}+\bar{r}_xu,
\]
which can also be seen directly from the encoding of the negated signed
digit representation.
\begin{observation}
  If $(x,r_x)$ is the canonical RN-encoding of a value $X$, then
  $(\bar{x},\bar{r}_x)$ is the canonical RN-encoding of $-X$, where
  $\bar{x}$ is the 1's~complement of $x$. Hence negation of a canonically
  encoded value is a constant time operation.
\end{observation}

It is important to note that although from a ``value perspective'' the
representation is redundant ($\val(a,1)=\val(a+u,0)$), it is not so when
considering the signed-digit representation. In this interpretation the
sign of the least significant digit carries information about the sign of
the part which possibly has been rounded away.
\begin{lemma}
\label{tail-sign}
Provided that a RN-represented number with canonical encoding $(a,r_a)$ is
non-zero, then $r_a=1$ implies that the least significant non-zero digit in
its signed-digit representation is~$1$ (the number was possibly rounded
up), and $r_a=0$ implies it is~$-1$ (the number was possibly rounded down).
\end{lemma}
\begin{proof}
  The result is easily seen when listing the trailing bits of the
  2's~complement representation of $2a+r_a$ (with $r_a$ in boldface) above
  those of $a$ together with the signed-digit representation:
\[\setlength{\arraycolsep}{2pt}
\renewcommand{\arraystretch}{1.2}
\begin{array}{ccccccc}
\dots &0&1&1&\dots&\bf1\\ 
\dots &.&0&1&\dots&1\\\hline
\dots &.&1&0&\dots&0\\ 
\end{array}\qquad
\begin{array}{ccccccc}
\dots &1&0&0&\dots&\bf0\\ 
\dots &.&1&0&\dots&0\\\hline
\dots &.&\bar1&0&\dots&0
\end{array}
\]\\[-4ex]
\end{proof}

\subsection{The Range of $p$-bit Canonically Encoded Numbers}

With $p+1$ bits in the tuple $(a,r_a)$, where $a$ is a $p$-bit
2's~complement representation paired with the round-bit $r_a$, it is just a
non-redundant encoding of a value represented by $p$ digits over the digit
set $\{-1,0,1\}$. Assuming that the radix point is at the rightmost end,
i.e., the numbers represent integer values, then the maximal value
representable in $p$ digits in $\{-1,0,1\}$ is $10\cdots0\sim2^{p-1}$ and
the minimal value is $\bar10\cdots0\sim-2^{p-1}$, hence the range is
sign-symmetric. The corresponding canonical encodings employing
2's~complement are $(011\cdots1,1)\sim(2^{p-1}-1,1)$ and
$(100\cdots0,0)\sim(-2^{p-1},0)$ respectively.

\subsection{An Alternative Interpretation}

But we may also interpret the representation $(a,r_a)$ as an interval
$\inv(a,r_a)$ of length $u/2$:
\begin{equation}
\txt\label{interval-interp}
\inv(a,r_a)=\left[a+r_a\frac{u}2\;;\;a+(1+r_a)\frac{u}2\right],
\end{equation}
when interpreting it as an interval according to what may have been thrown
away when rounding by truncation. For a detailed discussion of this
interpretation see \cite{KMP11}.

Hence even though $(a,1)$ and $(a+u,0)$ represent the same value $a+u$, as
intervals they are essentially disjoint, except for sharing a single
point. In general we may express the interval interpretation as pictured in
Fig.~\ref{fig.RN-intervals}
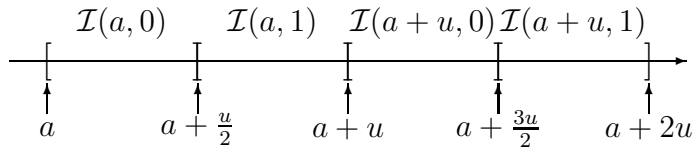
\begin{figure}[h]
\begin{center}
\setlength{\unitlength}{1mm}
\begin{picture}(80,10)(0,6)
\put(-5,11){\vector(1,0){90}}
\put(0,4){\vector(0,1){4}}\put(0,2){\mb{$a$}}
\put(20,4){\vector(0,1){4}}\put(20,2){\mb{$a+\frac{u}2$}}
\put(40,4){\vector(0,1){4}}\put(40,2){\mb{$a+u$}}
\put(60,4){\vector(0,1){4}}\put(60,2){\mb{$a+\frac{3u}2$}}
\put(80,4){\vector(0,1){4}}\put(80,2){\mb{$a+2u$}}
{\large
\put(0,11){\mb{[}}
\put(20,11){\mb{]}}
\put(20,11){\mb{[}}
\put(40,11){\mb{]}}
\put(40,11){\mb{[}}
\put(60,11){\mb{]}}
\put(60,11){\mb{[}}
\put(80,11){\mb{]}}
}
\put(10,16){\mb{$\inv(a,0)$}} 
\put(30,16){\mb{$\inv(a,1)$}}
\put(50,16){\mb{$\inv(a+u,0)$}}
\put(70,16){\mb{$\inv(a+u,1)$}}
\end{picture}
\end{center}
\caption{Binary Canonical RN-representations as Intervals}
\label{fig.RN-intervals}
\end{figure}

We do not intend to define an interval arithmetic, but only require that
the interval representation of the result of an arithmetic operation
$\odot$ satisfies\footnote{Note that this is the reverse inclusion of
that required for ordinary interval arithmetic}

\begin{equation}\label{int-condition}
\inv(A \odot B) \subseteq\inv(A)\odot\inv(B)=\{a \odot b|a\in A, b\in B\}.
\end{equation}

\section{Arithmetic Operations on RN-Represented Values}

We will here briefly summarize from \cite{KMP11} the realization of
addition and multiplication on fixed-point representations for fixed value
of $u$, but expand on the implementation of multiplication and division.
We want to operate directly on the components of the encoding $(a,r_a)$,
not on the signed-digit representation.

\subsection{Addition of RN-Represented Values}

Employing the value interpretation of encoded operands $(a,r_a)$ and
$(b,r_b)$ we have for addition:
\[
\begin{array}{rcl}
        \val(a,r_a) &=& a + r_au \\
   + \; \val(b,r_b) &=& b + r_bu \\ \hline
\val(a,r_a)+\val(b,r_b) &=& a+b+(r_a+r_b)u\\
\end{array}\label{xx}
\]

The resulting value has two possible representations, depending on the
choice of the rounding bit of the result. To determine what the rounding
bit of the result should be, we consider interval interpretations
(\ref{interval-interp}) of the two possible representations of the result.

To define the addition operator $\oplus$ on canonical encodings, we want
$\inv((a, r_a) \oplus (b,r_b)) \subseteq \inv(a, r_a) + \inv(b, r_b)$, and
$(a,r_a) \oplus (0,0) = (a,r_a)$, hence in order to keep addition
symmetric, we define addition of RN encoded numbers as follows.

\begin{definition}[Addition]
If $u$ is the unit in the last place of the operands, let:
\[
(a,r_a)\oplus(b,r_b) = ((a+b+(r_a\wedge r_b)u),r_a\vee r_b)
\]
where $r_a\wedge r_b$ may be used as carry-in to the 2's~complement
addition.\label{add-def}
\end{definition}

Recalling that $-(x, r_x) = (\bar{x}, \bar{r_x})$, we observe that using
Definition~\ref{add-def} for subtraction yields $(x,r_x) \ominus (x,r_x)
= (-u, 1)$, with $\val(-u, 1) = 0$. It is possible alternatively to define
addition on RN-encoded numbers as $(a, r_a) \oplus_2 (b, r_b) =
((a+b+(r_a\vee r_b)u),r_a\wedge r_b)$. Using this definition, $(x, r_x)
\ominus_2 (x, r_x) = (0,0)$, but then the neutral element for addition is
$(-u, 1)$, i.e., $(x, r_x) \oplus_2 (-u, 1) =(x,r_x)$.  

\subsection{Multiplying  RN-Represented Values}

By definition we have for the value of the product
\[
\begin{array}{rcl}
        \val(a,r_a) &=& a + r_au \\
        \val(b,r_b) &=& b + r_bu \\ \hline
\val(a,r_a)\val(b,r_b)
                &=& ab + (ar_b+br_a)u +r_ar_bu^2,
\end{array}
\]
noting that the unit of the result is $u^2$, assuming that $u\le 1$.  Using
the interval interpretation it turns out (for details see \cite{KMP11})
that we do not get proper interval inclusions for all sign combinations.
However, since negation of canonical (2's~complement) RN-encoded values can
be obtained by constant-time bit inversion, multiplication of such operands
can be realized by multiplication of the absolute values of the operands,
the result being supplied with the correct sign by a conditional inversion.
Thus employing bit-wise inversions, multiplication in canonical RN-encoding
may be handled like sign-magnitude multiplication, hence assuming that both
operands are non-negative:

\begin{definition}[Multiplication]
If $u$ is the unit in the last place, with $u\leq 1$, we define for
non-negative operands:
\[
(a,r_a)\otimes(b,r_b)=\left(ab+u(ar_b+br_a)
,r_a \wedge r_b\right),
\]
and for general operands by appropriate sign inversions of the operands and
result. If $u < 1$ the unit is $u^2 < u$ and the result may often have
to be rounded to unit $u$, which can be done by truncation.
\end{definition}

The product can be returned as $(p,r_p)$ with
$p=ab+u(ar_b+br_a)=a(b+r_bu)+br_au$, where the terms of $br_au$ may be
consolidated into the array of partial products.

\begin{example}
  For a 5-bit integer example let $(a,r_a)=(a_4a_3a_2a_1a_0,r_a)$ and
  $(b,r_b)=(b_4b_3b_2b_1b_0,r_b)$, or in signed-digit $b=d_4d_3d_2d_1d_0$,
  $d_i\in\{-1,0,1\}$, we note that $a_4=b_4=0$ since $a\ge0$ and
  $b\ge0$. It is then possible to consolidate the terms of $br_au$ (shown
  framed) into the array of partial products:\footnote{not showing the
    possible rewriting of negative partial products, requiring an
    additional row.}
\[\setlength{\arraycolsep}{2pt}
\begin{array}{ccccccccc|c}
   &   &   &   &   &    a_3  &   a_2  &   a_1  &  a_0  &\\\hline
   &   &   &   &    & a_3d_0 & a_2d_0 & a_1d_0 & a_0d_0 & d_0\\
   &   &   &   & a_3d_1 & a_2d_1 & a_1d_1 & a_0d_1 &\fbox{$b_0r_a$}& d_1\\
   &   &   & a_3d_2 & a_2d_2 & a_1d_2 & a_0d_2 &\fbox{$b_1r_a$}&    & d_2\\
   &   & a_3d_3 & a_2d_3 & a_1d_3 & a_0d_3 &\fbox{$b_2r_a$}&   &    & d_3\\
   & a_3d_4 & a_2d_4 & a_1d_4 & a_0d_4 &\fbox{$b_3r_a$} &  & &  & d_4\\[3pt]\hline
 p_8\;\,&  p_7 &    p_6   &   p_5  &  p_4   &  p_3  &  p_2 & p_1 & p_0
\end{array}
\]
thus the product is $(p,r_p)$ with $p=ab+u(ar_b+br_a)=a(b+r_bu)+br_au$ and
$r_p=r_a r_b$.

In particular for $a=(01011,1)$ and $b=(01001,1)\sim 1\bar1010$:
\[\setlength{\arraycolsep}{2pt}
\begin{array}{*{9}{>{\centering}p{3ex}}|r}
   &   &   &   &   & 1 & 0 & 1 & 1 &\\\hline
   &   &   &   &   & 0 & 0 & 0 & 0 & 0\\
   &   &   &   & 1 & 0 & 1 & 1 & \fbox1 & 1\\
   &   &   & 0 & 0 & 0 & 0 & \fbox0 &   & 0\\
   &   &$-1$ & 0 &$-1$ &$-1$ & \fbox0 &   &   & -1\\
   & 1 & 0 & 1 & 1 & \fbox1 &   &   &   & 1\\[1pt]\hline
 0 & 0& 1 & 1 & 1 & 0 & 1 & 1 & 1
\end{array}
\]
hence $(01011,1)\otimes(01001,1)=(001110111,1)$, where we note that
$(001110111,1)$ corresponds to the interval
$\left[01110111.1\;;\;01111000.0\right]$, clearly a subset of the 
interval
\begin{eqnarray*}
[01011.1\times 01001.1\;\;;01100\times 01010]\\
=[01101101.01\;;01111000.00].
\end{eqnarray*}
~\\[-6ex]
\end{example}

Thus multiplication of RN-represented values can be implemented on their
canonical encodings at about the same cost as ordinary 2's~complement
multiplication. Note that when recoding the multiplier into a higher radix
like $4$ and $8$, similar kinds of consolidation may be applied.

\subsection{Dividing RN-Represented Values}

As for multiplication we assume that negative operands have been
sign-inverted, and that the signs are treated separately. Employing our
interval interpretation (\ref{interval-interp}), to satisfy our interval
inclusion condition (\ref{int-condition}) we must require the result
of dividing $(x,r_x)$ by $(y,r_y)$ to be in the interval:
\[
\left[
\frac{x+r_x\frac{u}2}{y+(1+r_y)\frac{u}2} \;;\;
\frac{x+(1+r_x)\frac{u}2}{y+r_y\frac{u}2}
\right],
\]
where it is easily seen that the rational value\footnote{We could also have
  chosen to evaluate the quotient $\frac{x+r_x{u}}{y+r_y{u}}$. However
  dividing $(x,r_x)$ by the neutral element $(1,0)$ would then yield the
  result $(x+r_xu,0)$, whereas with the chosen quotient the result becomes
  $(x,r_x)$. The relative difference between these two expressions
  evaluated to some precision $p$ is at most $\ulp(p)/2$.}
\[
q=\frac{x+r_x\frac{u}2}{y+r_y\frac{u}2}
\]
belongs to that interval. Note that the dividend and divisor to obtain the
quotient $q$ are then constructed simply by appending the round bits to the
2's~complement parts, i.e., simply using the ``extended'' bit-strings as
operands.  To determine the accuracy needed in an approximate quotient
$q'=q+\eps$ consider the requirement
\begin{equation}
\label{div-bounds}
\frac{x+r_x\frac{u}2}{y+(1+r_y)\frac{u}2} 
< q+\eps
<
\frac{x+(1+r_x)\frac{u}2}{y+r_y\frac{u}2}.
\end{equation}

Generally division algorithms require that the operands are scaled, hence
assume that the operands satisfy $1\le x <2$ and $1\le y <2$, implying
$\frac12<q<2$. Furthermore assume that $x$ and $y$ have $p$ fractional
digits, so $u=2^{-p}$. To find sufficient bounds on the error $\eps$ in
(\ref{div-bounds}) consider first for $\eps\ge0$ the right bound. Here we
must require
\[\txt
(x+r_x\frac{u}2)+\eps(y+r_y\frac{u}2) < x+(1+r_x)\frac{u}2
\quad\mbox{or}\quad
\eps(y+r_y\frac{u}2) < \frac{u}2,
\]
which is satisfied for $\eps<\frac{u}4$, since $y+r_y\frac{u}2<2$. For the
other bound (for negative $\eps$) we must require
\[
\frac{x+r_x\frac{u}2}{y+(1+r_y)\frac{u}2}<
\frac{x+r_x\frac{u}2}{y+r_y\frac{u}2}+\eps
\]
or
\[\txt
-\eps(y+(1+r_y)\frac{u}2) < (x+r_x\frac{u}2)\frac{u}2,
\]
which is satisfied for $-\eps<\frac{u}4$, since $x\ge1$ and $y+u\le2$.

Hence $|\eps|<\frac{u}4$ assures that (\ref{div-bounds}) is satisfied, and
any standard division algorithm may be used to develop a binary
approximation to $q$ with $p+2$ fractional bits, $x=q'y+r$ with
$|r|<y2^{-p-2}$. Note that since $q$ may be less than 1, a left shift may
be required to deliver a $p+1$ signed-digit result, the same number of
digits as in the operands. Hence a bit of weight $2^{-p-1}$ will be
available as a (preliminary) round bit.

The sign of the remainder determines the sign of the tail beyond the bits
determined. Recall from Lemma~\ref{tail-sign} that when the round bit is~1,
the error is assumed non-positive, and non-negative when the round bit
is~0. If this is not the case then the resulting round bit must be
inverted, hence rounding is also here a constant time operation.

Similarly, function evaluations like squaring, square root and even the
evaluation of ``well behaved'' trans\-cendental functions may be defined
and implemented, just considering canonical RN-represented operands as
2's~complement values with a ``carry-in'' not yet absorbed, possibly using
interval interpretation to define the resulting round bit.

\section{A Floating Point Representation}

For an implementation of a binary floating point arithmetic unit (FPU) it
is necessary to define an encoding of an operand $(2^em,r_m)$, based on the
canonical encoding of the significand part (say $m$ encoded in $p+1$ bits,
2's~complement), supplied with the round bit $r_m$ and the exponent $e$ in
some biased binary encoding. It is then natural to pack the components into
a computer word (32, 64 or 128 bits), employing the same principles as used
in the IEEE-754 standard \cite{IEEE08} (slightly modified from what was
sketched in \cite{KMP11}). For normal values it is also here possible to
use a ``hidden bit'', noting that the 2's~complement encoding of the
normalized significand will have complementary first and second bits. Thus
representing the leading bit as a (separate) sign bit, the next bit (of
weight $2^0$) need not be represented and can be used as
``hidden-bit''. Hence let $f_m$ be the fractional part of the significand
$m$, assumed to be normalized such that $1\le |m| <2$, and let $s_m$ be the
sign-bit of $m$. The ``hidden bit'' is then the complement $\bar{s}_m$ of
the sign-bit. The components can then be allocated in the fields of a word
as:
\[
\setlength{\unitlength}{0.7mm}
\begin{picture}(120,0)
\put(0,0){\framebox(6,5){$s_m$}}
\put(6,0){\framebox(24,5){$e$}}
\put(30,0){\framebox(83,5){$f_m$}}
\put(113,0){\framebox(7,5){$r_m$}}
\end{picture}
\]
with the round bit in immediate continuation of the significand part. The
exponent $e$ can be represented in biased form as in the IEEE-754
standard. The number of bits allocated to the individual fields may be
chosen as in the different IEEE-754 formats, of course with the combined
$f_m,r_m$ together occupying the fraction field of those formats. The value
of a floating point number encoded this way can then be expressed as:
\[
2^{e-bias}\left([s_m\bar{s}_m.f_1f_2\cdots f_{p-1}]_{2c}+r_m2^{-p-1}\right)
\]
where $f_1,f_2,\cdots,f_{p-1}$ are the (fractional) bits of $f_m$.

Subnormal and exceptional values may be encoded as in the IEEE-754
standard, noting that negative subnormals have leading ones. Observe that
the representation is sign-symmetric and that negation is obtained by
inverting the bits of the significand.

We shall now discuss how the fundamental operations may be implemented on
such floating point RN-representations, not going into details on overflow,
underflow and exceptional values, as these situations can be treated
exactly as known for the standard binary IEEE-754 representation.  Note
again that we want directly to operate on the 2's~complement encoding using
$s_m,f_m,r_m$, not on the signed-digit representation. Right-shifts are
trivial, but for alignment of operands or left normalizing results, we must
investigate how in general we can perform left shifts using the
2's~complement encoding.

Thinking of the value as represented in binary signed-digit, zeroes have to
be shifted in when left shifting.  In our encoding, say for a positive
result $(d,r_d)$ we may have a 2's~complement bit pattern:
\[
 d\sim 0\;0\;\cdots\;0\;1\;b_{k}\cdots b_{p-1}\mbox{ and round bit }r_d
\]
to be left-shifted. Here the least significant signed digit is encoded as
$\left\{{r_d}\atop {b_{p-1}}\right\}$. Zero-valued digits to be shifted in
may then be encoded as $\left\{{r_d}\atop {r_d}\right\}$, as confirmed from
applying the addition rule for obtaining $2\times(x,r_x)$ by
$(x,r_x)\oplus(x,r_x)=(2x+r_xu,r_x)$.

It then follows that shifting in bits of value $r_d$ will precisely achieve
the effect of shifting in zeroes in the signed-digit interpretation:
\[
2^k d\sim 0\;1\;b_{k}\cdots b_{p-1} r_d\cdots r_d\mbox{ with round bit }r_d.
\]

\subsection{Multiplication and Division}

Since the exponents are handled separately, forming the product or quotient
of the significands is precisely as described previously for fixed point
representations: sign-inverting negative operands by bitwise inversion,
forming the product or quotient, possibly normalizing and rounding it, and
supplying it with the proper sign by negating the result if the operands
were of different signs.

\subsection{Addition}

Before addition or subtraction there is in general a need of alignment of
the two operand significands, according to the difference of their
exponents (too large a difference is treated as a special case, see
below). The operand having the larger exponent must be left-shifted, with
appropriate digit values appended at the least significant end, to overlap
with the significand of the smaller operand. In effective subtractions,
after cancellation of leading digits it may be necessary to left-normalize.
Addition is traditionally now handled in an FPU as two cases \cite{Far81a},
where the ``near case'' is dealing with effective subtraction of operands
whose exponents differ by no more than one, where alignment is a constant
time operation.

\subsubsection{Subtraction, the "near case"}

Here a significant cancellation of leading digits may
occur, and thus a variable amount of normalization shifts on the result are
required, handled by shifting in copies of the round-bit.
Figure~\ref{fig-near-path} shows a possible pipelined implementation of
this case, where $\lzd(d)$ is a log-time algorithm for ``leading zeroes
determination'' of the difference (see e.g., \cite{Kor09}) to determine the
necessary normalization shift amount. This determination is based on a
redundant representation of the difference (obtained in constant time by
pairing the aligned operands), taking place in parallel with the
2's~complement subtraction (conversion from redundant to non-redundant
representation). Normalization can then take place on the non-redundant
difference without need for sign inversion.

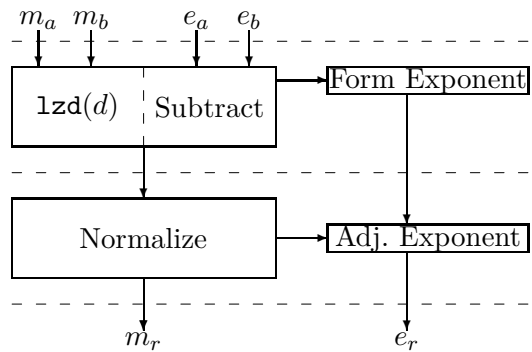
\begin{figure}[hbt]
\begin{center}\small
\setlength{\unitlength}{0.7mm}
\begin{picture}(120,53)(15,50)

\put(35,102){\vector(0,-1){7}}\put(35,104){\mb{$m_a$}}
\put(45,102){\vector(0,-1){7}}\put(45,104){\mb{$m_b$}}

\put(65,102){\vector(0,-1){7}}\put(65,104){\mb{$e_a$}}
\put(75,102){\vector(0,-1){7}}\put(75,104){\mb{$e_b$}}

\dashline{2}(30,100)(128,100)

\put(30,80){\framebox(50,15){}}

\put(55,80){\vector(0,-1){10}}

\put(42.5,87.5){\mb{$\lzd(d)$}}
\put(67.5,87.5){\mb{Subtract}}
\put(55,80){\dashline{2}(0,0)(0,15)}

\put(90,90){\framebox(38,5){Form Exponent}}
\put(80,92.5){\vector(1,0){10}}
\put(105,90){\vector(0,-1){25}}

\dashline{2}(30,75)(128,75)

\put(30,55){\framebox(50,15){Normalize}}
\put(55,55){\vector(0,-1){10}}

\put(90,60){\framebox(38,5){Adj.\ Exponent}}
\put(80,62.5){\vector(1,0){10}}

\put(55,55){\vector(0,-1){10}}\put(55,43){\mb{$m_r$}}

\put(105,60){\vector(0,-1){15}}\put(105,43){\mb{$e_r$}}

\dashline{2}(30,50)(128,50)
\end{picture}
\end{center}

\caption{Near Path, effective subtraction when
  $|e_a-e_b|\le1$}\label{fig-near-path}
\end{figure}
For simplicity in the figure we assume that $m_a,m_b$ and $m_r$ are the
2's~complement operands, respectively the result, together with their
appended round-bits.

\subsubsection{Addition, the "far case"}

The remaining cases dealt with are the situations where the result of
adding or subtracting the aligned significands at most requires
normalization by a single right or left shift. Since negation is a constant
time operation we may assume that an effective 2's~complement addition is
to be performed of the left-aligned larger operand (appended with copies of
the round-bit) and the sign-exended smaller operand. Rounding can then be
performed as usual by truncation, noting that here there are only two
log-time operations, the variable amount of alignment shifts and the
addition.  Compared with the IEEE-754 standard number representation the
``expensive'' determination of a sticky bit and rounding incrementation is
avoided.  Figure~\ref{fig-far-path} shows a possible two-stage pipeline
implementation.

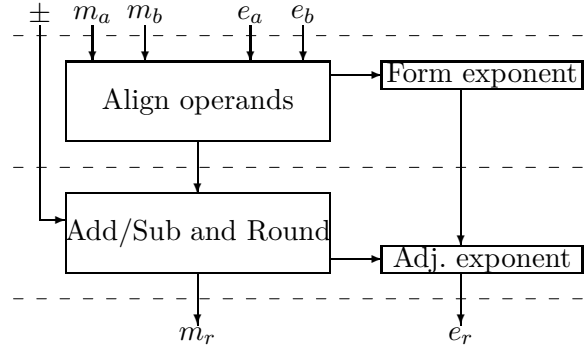
\begin{figure}[ht]
\begin{center}\small
\setlength{\unitlength}{0.7mm}
\begin{picture}(120,53)(15,50)
\put(25,102){\line(0,-1){37}}\put(25,104){\mb{$\pm$}}
\put(25,65){\vector(1,0){5}}

\put(35,102){\vector(0,-1){7}}\put(35,104){\mb{$m_a$}}
\put(45,102){\vector(0,-1){7}}\put(45,104){\mb{$m_b$}}

\put(65,102){\vector(0,-1){7}}\put(65,104){\mb{$e_a$}}
\put(75,102){\vector(0,-1){7}}\put(75,104){\mb{$e_b$}}

\dashline{2}(20,100)(128,100)

\put(30,80){\framebox(50,15){Align operands}}
\put(55,80){\vector(0,-1){10}}

\put(90,90){\framebox(38,5){Form exponent}}
\put(80,92.5){\vector(1,0){10}}
\put(105,90){\vector(0,-1){30}}

\dashline{2}(20,75)(128,75)

\dashline{2}(20,50)(128,50)

\put(30,55){\framebox(50,15){Add/Sub and Round}}
\put(55,55){\vector(0,-1){10}}\put(55,43){\mb{$m_r$}}

\put(90,55){\framebox(38,5){Adj.\ exponent}}
\put(80,57.5){\vector(1,0){10}}
\put(105,55){\vector(0,-1){10}}\put(105,43){\mb{$e_r$}}

\end{picture}
\end{center}
\caption{Far Path, add or subtract when $|e_a-e_b|\ge2$}
\label{fig-far-path}
\end{figure}


In the case where the exponent difference exceeds the number of operand
bits, it is not necessary to form the exact sum. The result can be
constructed from the 2's~complement significand of the larger operand,
supplied with a round-bit obtained by a very simple rule providing a result
obeying the interval inclusion condition (\ref{int-condition}). Assuming
that the smaller operand is to be added, simply force the round-bit of the
result to become equal to the complemented sign-bit of the smaller operand
(and of course in case of subtraction the sign-bit).

To see this, consider the interval interpretation (\ref{interval-interp})
of the significand of the larger operand ${\cal I}\left(a,r_a\right)$
together with a bounding interval for the smaller operand
$\left[0\;;\frac{u}2\right]$ (assumed positive), where $u=\ulp(a)$ is the
unit of the least-significant position of the larger operand:
\begin{eqnarray*}
\txt{\cal I}\left(a,r_a\right)+\left[0\;;\frac{u}2\right]
& = &\txt\left[a+r_a\frac{u}2\;;a+(1+r_a)\frac{u}2\right]
+\left[0\;;\frac{u}2\right]\\
& = &\txt\left[a+r_a\frac{u}2\;;a+(2+r_a)\frac{u}2\right]\\
& = &\left\{
\begin{array}{ccl}
\left[a\;;a+u\right] & \mbox{for} & r_a=0\\
\left[a+\frac{u}2\;;a+\frac32u\right] & \mbox{for} & r_a=1.
\end{array}
\right.
\end{eqnarray*}
Hence chosing the result as ${\cal I}\left(r,r_a\right) = {\cal
  I}\left(a,1\right) = \left[a+\frac{u}2\;;a+u\right]$ it satisfies the
interval condition~(\ref{int-condition}) when the smaller operand is
positive. Similarly, if the smaller operand is negative
\begin{eqnarray*}
\txt{\cal I}\left(a,r_a\right)+\left[-\frac{u}2;0\right]
& = &\left\{
\begin{array}{ccl}
\left[a-\frac{u}2\;;a+\frac{u}2\right] & \mbox{for} & r_a=0\\
\left[a\;;a+u\right] & \mbox{for} & r_a=1
\end{array}
\right.
\end{eqnarray*}
hence the result can be chosen as ${\cal I}\left(r,r_a\right) = 
{\cal I}\left(a,0\right) = \left[a\;;a+\frac{u}2\right]$. In summary we
have:
\begin{lemma}
  Given RN-represented floating point operands $A=(s_a,e_a,f_a,r_a)$
  and $B=(s_b,e_b,f_b,r_b)$ with $p$-bit significands satisfying
  $e_a>e_b+p$, then the result of the addition $S=A+B$ can be represented
  as $S=(s_a,e_a,f_a,\bar{s}_b)$.
\end{lemma}

\subsection{Discussion of the Floating Point RN-representation}
As seen above it is straightforward to define binary floating point
representations, when the significand is encoded in the canonical
2's~complement encoding with the round-bit appended. An FPU implementation
of the basic arithmetic operations is feasible in about the same complexity
as one based on the sign-magnitude representation of the IEEE-754 standard
for binary floating point. But since the round-to-nearest functionality is
achieved at less hardware complexity, the arithmetic operations will
generally be faster, by avoiding the usual log-time ``sticky-bit''
determination and rounding incrementation. 

Negation is obtained in constant time by bit-wise inversion, noting that
the domain of representable values is sign symmetric.  Although one less
bit is used for the significand, the round-bit provides additional
information such that the discretization error is the same as in the
IEEE-754 representation of the compatible format.  Notice that if the
result is not exact, then the round-bit provides information about which
direction the rounding took. In effect the round-bit provides the same
information on the accuracy of the result as the additional bit available
in the binary IEEE-754 encoding of the significand. Just as in an x86 FPU
it is possible to signal exactness of a result.

The directed roundings can also be realized at minimal cost; however,
requiring the calculation of a ``sticky bit'', as also needed for the
directed roundings of the IEEE-754 representation, but no rounding
incrementation is needed here:
\begin{theorem}
\label{directed-roundings}
Let a number after truncation of tail $t$ have encoding $(a,r_a)$ with
sign-bit $s_a$, then the directed roundings can be realized by changing the
resulting round-bit as follows:
\[
\begin{array}{ll}
\ru: & r_a:=1\\
\rd: & r_a:=0\\
\rz: & r_a:=s_a\\
\ra: & r_a:=\bar{s}_a,
\end{array}
\]
conditional on the truncated tail $t$ (the ``sticky-bit'') being non-zero.
\end{theorem}
\begin{proof}
  Consider the case of $\ru$ when $r_a=0$. By Lemma~\ref{tail-sign} the
  least significant non-zero signed-digit of the truncated $(a,r_a)$ is
  $-1$, and if $t\ne0$ the value was effectively rounded down, thus $r_a$
  should be changed to $r_a=1$, whereas it should not be changed when
  $r_a=1$. The other cases follow similarly.
\end{proof}

\section{Conclusions and Discussion}

Concentrating on binary RN-represented operands over the signed digit set
$\{-1,0,1\}$, allowing trivial (constant time) rounding by truncation, 
we have previously proposed a simple encoding based on the ordinary
2's~complement representation, with negation also being a constant time
operation, which often simplifies the implementation of arithmetic
algorithms. Operands in the canonical encoding can be used directly at
hardly any penalty in the implementation of the basic arithmetic
operations, e.g., addition, subtraction, multiplication and division,
allowing constant time rounding. Thus despite the RN-representation encodes
a very special signed-digit representation, it allows the operations to be
performed in a slightly modified 2's~complement arithmetic.

The fixed point encoding immediately allows for the definition
of corresponding floating point representations, which in a comparable
hardware FPU implementation will be simpler and faster than an equivalent
IEEE-754 standard conforming implementation.

The particular feature that rounding-to-nearest is obtained by truncation,
implies that repeated roundings ending in some lower precision yields the
same result, as if a single rounding to that precision was performed. In
\cite{Lee89} it was proposed to attach some state information (2 bits) to a
rounded result, allowing subsequent roundings to be performed in such a way
that these problems are avoided. It was shown that this property holds for
any specific IEEE-754 rounding mode, including in particular for the
round-to-nearest-even mode. But the IEEE-754 roundings may still require
log-time incrementations, which are avoided with the proposed
RN-representation.

Thus in applications where conformance to the IEEE-754 standard is not
required, employing the proposed floating-point RN-representation, it is
possible to avoid the penalty of log-time roundings. Signal processing may
be an application area where specialized hardware (ASIC or FPGA) is often
used anyway, where the RN-representation can provide faster arithmetic with
un-biased round-to-nearest operations at reduced area and delay.

\bibliographystyle{alpha} {\small

\begin{thebibliography}{NMLE00}

\bibitem[Boo51]{Boo51}
A.D. Booth.
\newblock {A Signed Binary Multiplication Technique}.
\newblock {\em Q.~J.~Mech.~Appl.~Math.}, 4:236--240, 1951.


\bibitem[Far81]{Far81a}
P.M. Farmwald.
\newblock {\em {On the Design of High Performance Digital Arithmetic Units}}.
\newblock PhD thesis, Stanford, Aug. 1981.

\bibitem[IEE08]{IEEE08}
IEEE.
\newblock {\em IEEE Std.~754\texttrademark-2008 Standard for Floating-Point
  Arithmetic}.
\newblock IEEE, 3 Park Avenue, NY 10016-5997, USA, August 2008.

\bibitem[KM05]{KM05b}
P.~Kornerup and J.-M. Muller.
\newblock {RN-Coding of Numbers: Definition and some Properties}.
\newblock In {\em Proc.~IMACS'2005}, July 2005.
\newblock Paris.

\bibitem[KMP11]{KMP11} P.~Kornerup, J-M. Muller, and A.~Panhaleux.
\newblock {Performing Arithmetic Operations on Round-to-Nearest
Representations}.  
\newblock {\em IEEE Transactions on Computers}, 60(2):282--291, 
February 2011.


\bibitem[Kor09]{Kor09}
P.~Kornerup.
\newblock {Correcting the Normalization Shift of Redundant Binary
  Representations}.
\newblock {\em IEEE Transactions on Computers}, 58(10):1435--1439, October
  2009.

\bibitem[Lee89]{Lee89}
C.~Lee.
\newblock {Multistep Gradual Rounding}.
\newblock {\em IEEE Transactions on Computers}, 38(4):595--600, April 1989.


\end{thebibliography}

\end{document}